\documentclass[a4paper,12pt,reqno]{amsart}
\usepackage[utf8x]{inputenc}
\usepackage{tikz}
\usetikzlibrary{decorations.markings}
\usepackage{amssymb,enumerate,mathtools}
\usepackage{amsmath}
\usepackage{amsthm,fontenc}
\usepackage{hyperref}
\usepackage{amsrefs}
\newtheorem{theorem}{Theorem}[section]
\newtheorem{lemma}[theorem]{Lemma}

\newtheorem{corollary}[theorem]{Corollary}

\newtheorem*{theorem*}{Theorem}
\DeclarePairedDelimiter{\ceil}{\lceil}{\rceil}
\DeclarePairedDelimiter{\floor}{\lfloor}{\rfloor}
\newcommand{\be}{\begin{equation}}
\newcommand{\ee}{\end{equation}}
\newcommand{\ba}{\begin{eqnarray*}}
\newcommand{\ea}{\end{eqnarray*}}
\newcommand{\pPell}{p_k}

\newcommand{\Mod}[1]{\ (\mathrm{mod}\ #1)}
\newcommand\Item[1][]{%
  \ifx\relax#1\relax  \item \else \item[#1] \fi
  \abovedisplayskip=0pt\abovedisplayshortskip=0pt~\vspace*{-\baselineskip}}
\DeclarePairedDelimiter{\norm}{\lVert}{\rVert}

\title[Base-$q$ and Ostrowski sum-of-digits functions]{Joint distribution in residue classes of the base-$q$ and Ostrowski digital sums}

\date{}
\author[Divyum Sharma]{Divyum Sharma}
\address{Department of Pure Mathematics, University of Waterloo, Ontario, Canada.}
\email{divyum.sharma\symbol{64}uwaterloo.ca}

\subjclass[2010]{11A63, 11A55, 11J71}
\keywords{Ostrowski representation, Sum of digits function, Joint distribution}
\begin{document}
	\begin{abstract}
 Let $q$ be an integer $\geq 2$ and let $S_q(n)$ denote the sum of digits of $n$ in base $q$.
For
	\[
	\alpha=[0;\overline{1,m}],\ m\geq 2,
	\]
let $S_{\alpha}(n)$ 
denote the sum of digits in the Ostrowski $\alpha$-representation of $n$. Let $m_1,m_2\geq 2$ be integers with $$\gcd(q-1,m_1)=\gcd(m,m_2)=1.$$ We prove that there exists $\delta>0$ such that for all integers $a_1,a_2$,
		\begin{align*}
		&|\{0\leq n<N: S_{q}(n)\equiv a_1\Mod{m_1},\ S_{\alpha}(n)\equiv a_2\Mod{m_2}\}|\\
		&=\frac{N}{m_1m_2}+O(N^{1-\delta}).
		\end{align*}
		The asymptotic relation implied by this equality was proved by Coquet, Rhin \& Toffin and the equality was proved for the case $\alpha=[\ \overline{1}\ ]$ by Spiegelhofer.
	\end{abstract}
\maketitle
 \section{Introduction}
 Let $q$ be an integer $\geq 2$ and let $S_q(n)$ denote the sum of digits of $n$ in base $q$. Gel$\cprime$fond \cite{Ge} proved that if $m$ is a positive integer coprime to $q-1$, then the function $S_q(n)$ is uniformly distributed modulo $m$. Further, he conjectured that if $q_1,q_2,m_1,m_2$ are integers $\geq 2$ with $\gcd(q_1,q_2)=\gcd(m_1,q_1-1)=\gcd(m_2,q_2-1)=1$, then 
  there exists $\delta=\delta(q_1,q_2,m_1,m_2)>0$ such that 
 		\begin{align*}
		&|\{0\leq n<N: S_{q_1}(n)\equiv a_1\Mod{m_1},\ S_{q_2}(n)\equiv a_2\Mod{m_2}\}|\\
		&=\frac{N}{m_1m_2}+O(N^{1-\delta})
		\end{align*}
for all integers $a_1,a_2$. The asymptotic relation implied in this conjecture was proved by B\'esineau \cite{Be}; while the conjecture was proved in its full strength by Kim \cite{Kim} (See also \cite{Rec} for some recent improvements of Kim's result). We mention that a different type of joint distribution related to sum of digits functions was studied by Solinas \cite{Sol}.

In this paper we are concerned with the above problem for the sum of digits functions of the base-$q$ and \textit{Ostrowski representation} of integers.
In 1922, Ostrowski \cite{Os} discovered a numeration system based on continued fractions. He showed that the sequence of the denominators of the convergents to the simple continued fraction expansion of an irrational number forms the basis for a numeration system. More precisely, he proved the following result.
\begin{theorem}\cite[Theorem 3.9.1]{AS}\label{Ostro}
Let $\alpha$ be an irrational real number having continued fraction expansion $[a_0;a_1,\ldots]$. Let $(q_n)_{n\geq 0}$ be the sequence of the denominators of the convergents to the continued fraction expansion. Then every non-negative integer $n$ can be expressed uniquely as
	\be\label{decomp}
	n=\sum\limits_{0\leq i\leq \ell} b_iq_i,
	\ee
where the $b_i$'s are integers satisfying
	\begin{enumerate}[(i)]
	\item $0\leq b_0<a_1$.
	\item $0\leq b_i\leq a_{i+1}$ for $i\geq 1$.
	\item For $i\geq 1$, if $b_i=a_{i+1}$, then $b_{i-1}=0$.
	\end{enumerate}
\end{theorem}
Note that Condition $(iii)$ above states that the recurrence relation $q_n=a_nq_{n-1}+q_{n-2}$ cannot be used to replace a linear combination of summands with another summand. The expression given in \eqref{decomp} is called the \textit{Ostrowski $\alpha$-representation} of $n$. See \cite{sur} for a survey on the connections between the Ostrowski numeration systems and combinatorics of words, and \cite{dyn} for a study of ergodic and topological-dynamical properties of various dynamical systems associated to Ostrowski $\alpha$-representations.
 
 If the $\alpha$-representation of a positive integer $n$ is given by
	\be\label{nLehmer}
	n=\sum\limits_{0\leq i\leq \ell} b_i(n)q_i,
	\ee
let 
	\[
	S_{\alpha}(n)=\sum\limits_{0\leq i\leq \ell} b_i(n)
	\]
be the sum of digits. In \cite{Co}, Coquet, Rhin \& Toffin studied the relation between the functions $S_q(n)$ and
 $S_{\alpha}(n)$. They proved the following theorem.
	\begin{theorem*}
	Let $q$ be an integer $\geq 2$ and let $\alpha$ be an irrational real number. The sequence $n\rightarrow xS_q(n)+yS_{\alpha}(n)$ is uniformly distributed modulo $1$ if and only if at least one of $x$ and $y$ is irrational.
	\end{theorem*}

In \cite{Spt}, Spiegelhofer considered the case when
$$\alpha=\frac{1+\sqrt{5}}{2}=[\overline{1}].$$
(Note that in this case, the sequence $(q_n)$ is the sequence of Fibonacci numbers and that every non-negative integer can be uniquely expressed as a sum of non-consecutive Fibonacci numbers. This representation is known as the \textit{Zeckendorf representation} of integers (see \cite{Ze}).) 
He proved that if $\theta\in\mathbb{R}$ and $\gamma\in\mathbb{R}\setminus\mathbb{Z}$, then
		\[
		\sum\limits_{n<N}e(\theta S_{q}(n)+\gamma S_{\alpha}(n))=O(N^{1-\delta})
		\]
	for some $\delta>0$. (Throughout this paper, $e(x)$ denotes $\exp(2\pi ix)$.)
As a consequence, he obtained
	\begin{theorem*}\cite[Corollary 5.3]{Spt}
	Let $\alpha=(1+\sqrt{5})/2$ and let $q,m_1,m_2$ be integers $\geq 2$ with $\gcd(m_1,q-1)=1$. There exists $\delta>0$ such that for all integers $a_1,a_2$,
		\begin{align*}
		&|\{0\leq n<N: S_{q}(n)\equiv a_1\Mod{m_1},\ S_{\alpha}(n)\equiv a_2\Mod{m_2}\}|\\
		&=\frac{N}{m_1m_2}+O(N^{1-\delta}).
		\end{align*}
	\end{theorem*}
In \cite{Fou}, Coquet, Rhin \& Toffin gave three sufficient conditions for the set
\[
|\{n\in\mathbb{N}: S_{q}(n)\equiv a_1\Mod{m_1},\ S_{\alpha}(n)\equiv a_2\Mod{m_2}\}
\]
to have asymptotic density equal to $1/(m_1m_2)$. One of these conditions is that the sequence $(q_k)$ be lacunary and $\gcd(a_k,m_2)$ be equal to one for infinitely many indices $k$. Note that this condition is satisfied for 
	\[
	\alpha=[0;\overline{1,m}]=\frac{-m+\sqrt{m^2+4m}}{2},\ m\geq 2
	\]
as 	
\[
\frac{q_{k+1}}{q_k}\geq\begin{cases}1+\frac{q_{k-1}}{q_k}\geq 1+\frac{1}{m+1} &\mbox{ if } a_{k+1}=1,\\
 m+\frac{q_{k-1}}{q_k}\geq m+\frac{1}{2} &\mbox{ if } a_{k+1}=m.
 \end{cases}
\]
For these values of $\alpha$, we improve the above asymptotic estimate to an estimate with error term $O(N^{1-\delta})$. Let $\norm{x}=\min\limits_{j\in\mathbb{Z}}|x-j|$. We prove
	\begin{theorem}\label{thm1}
	Let $q$ be an integer $\geq 2$ and let $$\alpha=[0;\overline{1,m}],\ m\geq 2.$$
	Let $\theta,\gamma\in\mathbb{R}$ with $\norm{m\gamma}\neq 0$. Then there exists $\delta>0$ such that
		\[
		\sum\limits_{n<N}e(\theta S_{q}(n)+\gamma S_{\alpha}(n))=O(N^{1-\delta}),
		\]
	where the $O$-constant depends only on $q$ and $m$.
	\end{theorem}
As a consequence, we obtain
	\begin{corollary}\label{cor1}
	Let $q$ and $\alpha$ be as in Theorem \ref{thm1} and let $m_1,m_2$ be integers $\geq 2$ with $\gcd(q-1,m_1)=\gcd(m,m_2)=1$. There exists $\delta>0$ such that for all integers $a_1,a_2$,
		\begin{align*}
		&|\{0\leq n<N: S_{q}(n)\equiv a_1\Mod{m_1},\ S_{\alpha}(n)\equiv a_2\Mod{m_2}\}|\\
		&=\frac{N}{m_1m_2}+O(N^{1-\delta}).
		\end{align*}
	\end{corollary}
The proof relies on Weyl and van der Corput's method. 
In Section \ref{pre}, we introduce some notation and record some preliminary lemmas. Following \cite{Spt}, we obtain a characterization of integers with the same initial digits in their Ostrowski $\alpha$-representations in Section \ref{lem_sec}. We then use it to obtain an analogue of inverse discrete Fourier transform in this case and also derive a uniform upper bound for the \textit{Fourier coefficients}. Finally, we prove Theorem \ref{thm1} and Corollary \ref{cor1} in Section \ref{last}.
 \section{Preliminaries}\label{pre}
 Let $\ceil{x}$ denote the smallest integer greater than or equal to $x$ and $\{x\}$ denote the fractional part of $x$.\\
 Since
	\[
	\alpha=[0;\overline{1,m}],
	\]
we have $q_0=q_1=1$ and
	\[
	q_n=\begin{cases}
		mq_{n-1}+q_{n-2} &\mbox{ if } n \textrm{ is even} \\
		q_{n-1}+q_{n-2} &\mbox{ if } n \textrm{ is odd}.
		\end{cases}
	\]
Let $d=m^2+4m$ and
	\be\label{phin}
	\varphi=\frac{m+2+\sqrt{d}}{2}.
	\ee
Then
	\begin{align}\label{qiLehmer}
		q_{2\ell}&=\frac{m+\sqrt{d}}{2\sqrt{d}}\varphi^{\ell}-\frac{m-\sqrt{d}}{2\sqrt{d}}\varphi^{-\ell},\\
\nonumber q_{2\ell+1}&=\frac{1}{\sqrt{d}}\varphi^{\ell+1}-\frac{1}{\sqrt{d}}\varphi^{-\ell-1}.
	\end{align}
By Theorem \ref{Ostro}, the digits in the $\alpha$-representation \eqref{nLehmer} of a positive integer $n$ satisfy
	\[
	b_{2\ell}(n)\leq 1\textrm{ and } b_{2\ell+1}(n)\leq m
	\]
 for all non-negative integers $\ell$. Given an integer $k\geq 1$,
let $t(n;k)$ denote the truncation of the sum in \eqref{nLehmer} after $k$ digits, i.e.
	\[
	t_{\alpha}(n;k)=\sum\limits_{0\leq i\leq k-1} b_i(n)q_i
	\]
and let $S_{\alpha,k}(n)$ denote the sum of digits up to $k$, i.e.
	\[
	S_{\alpha,k}(n)=\sum\limits_{0\leq i\leq k-1} b_i(n).
	\]
Next, let
 \[
 S_{q,t}(n)=S_q(n (\textrm{ mod } q^t)).
 \]
 Let $G_t(\ell)=G_t(\ell,\theta)$ denote the discrete Fourier coefficients of the function $e(\theta S_{q}(n))$, i. e.
	\[
	G_t(\ell,\theta)=\frac{1}{q^{t}}\sum\limits_{u<q^t}e(\theta S_q(u)-\ell u q^{-t}).
	\]
Then
	\begin{align}
	\label{dft1} e(\theta S_{q,t}(n))&=\sum\limits_{\ell<q^t} e(\ell n q^{-t})G_t(\ell,\theta),\\
	\label{dft2} e(-\theta S_{q,t}(n))&=\sum\limits_{\ell<q^t} e(\ell n q^{-t})\overline{G_t(-\ell,\theta)}.
	\end{align}
Note that, by Parseval's identity,
	\be\label{Parseval}
	\sum\limits_{\ell<q^t}|G_t(\ell)|^2=1.
	\ee
For negative integers $n$, we define
\[
S_q(n)=S_{\alpha}(n)=0.
\]
We now list some results needed in the proof. The following is an elementary lemma  on exponential sums.
\begin{lemma}\label{elem}
Let $x\in\mathbb{R}$ and $N,R\geq 0$. Then
\begin{enumerate}[(i)]
\item \cite[Lemma 1]{Kor}\	\[
	\Big|\sum\limits_{n<N}e(nx)\Big|\leq \min\left(N,\frac{1}{2\norm{x}}\right).
	\]
\item \[
	\sum\limits_{|r|<R}(R-|r|)e(rx)=\big|\sum\limits_{r<R}e(rx)\big|^2.
	\]
\end{enumerate}
\end{lemma}
We now record an estimate from \cite{Spt}, which is proved using discrepancy estimate for the sequence $(n\varphi)$, where $\varphi$ has bounded partial quotients. We use it for $\varphi$ as given in \eqref{phin}.
  \begin{lemma}\label{discrep}\cite[Lemma 5.8]{Spt}
  Let $I$ be a finite interval in $\mathbb{Z}$. Let $K$ and $a$ be real numbers with $K\geq 1$. Then
  	\[
  	\sum\limits_{h\in I}\min\left(K,\frac{1}{\norm{a+h\varphi}^2}\right)\ll \sqrt{K}\lambda(I)+K\ln\lambda(I).
  	\]
  	(Here, $\lambda$ denotes the Lebesgue measure on $\mathbb{R}$.)
   \end{lemma}
Next, we state the version of the Weyl-van der Corput inequality that will be used later.
  \begin{lemma}\label{corput}\cite[Lemma 2.5]{Gr}
  Let $z_0,\ldots,z_{N-1}$ be complex numbers. For all positive integers $R$, we have
  	\[
 	\Big|\sum\limits_{n=0}^{N-1}z_n\Big|^2\leq\frac{N+R-1}{R}\sum\limits_{|r|<R}\Big(1-\frac{|r|}{R}\Big)\Big|\sum\limits_{\substack{0\leq n< N\\0\leq n+r< N}}\overline{z}_nz_{n+r}\Big|.
 	\]
 	\end{lemma}
 The following lemma states that for \textit{most} integers $n$, the representations of $n$ and $n+r$ may
 differ at digits corresponding to the first few base elements only.
\begin{lemma}\label{trunc}
 Let $N,r,k,t$ be non-negative integers with $k\geq 2$ and let $\theta,\gamma$ be real numbers. Then
 \begin{align*}
 	(i)\ &|\{n<N:e(\theta S_{q}(n+r))\overline{e(\theta S_{q}(n))}\neq e(\theta S_{q,t}(n+r))\overline{e(\theta S_{q,t}(n))}\}|\\
 	&\leq\frac{Nr}{q^t}+r.
 	\end{align*}
 	
\begin{align*}
 	(ii)\ &|\{n<N:e(\gamma S_{\alpha}(n+r))\overline{e(\gamma S_{\alpha}(n))}\neq e(\gamma S_{\alpha,k}(n+r))\overline{e(\gamma S_{\alpha,k}(n))}\}|\\
 	&\leq\frac{Nr}{q_{k-1}}.
 	\end{align*}
 \end{lemma}
 \begin{proof}
 See \cite[Lemma 1.17]{Spt} for a proof of $(i)$ and \cite[Lemma 2.6]{Spp} for a proof of $(ii)$.
 \end{proof}
 \section{Lemmas}\label{lem_sec}
 We first derive a characterization (Corollary \ref{criteria}) of integers $n$ with the same value of $t_{\alpha}(n;k)$, for a given $k$. Later, we use this to obtain \textit{discrete Fourier transform} for the function $e(\gamma S_{\alpha,k}(n))$. This is analogous to \cite[Proposition 5.7 \& Proposition 5.4]{Spt}.
\begin{lemma}\label{crit1}
Let $k\geq 2$ be an integer. Let $\pPell(n)=(-1)^kn\varphi$, where $\varphi$ is as in \eqref{phin}. Define
	\begin{align*}
	A_k^{(1)}&=\begin{cases}
				\left[\frac{m-\sqrt{d}}{2\varphi^{k_0}},\frac{1}{\varphi^{k_0}}\right) &\mbox{ if } k=2k_0,\ k_0\in\mathbb{N},\\[10pt]
				\left[\frac{-1}{\varphi^{k_0+1}}, \frac{-m+\sqrt{d}}{2\varphi^{k_0}}\right) &\mbox{ if } k=2k_0+1,\ k_0\in\mathbb{N}
			    \end{cases}\\
	A_k^{(2)}&=\begin{cases}
				\left[\frac{m-\sqrt{d}}{2\varphi^{k_0}},\frac{1}{\varphi^{k_0+1}}\right) &\mbox{ if } k=2k_0,\ k_0\in\mathbb{N},\\[10pt]
				\left[\frac{-1}{\varphi^{k_0+1}}, \frac{-m-1+\sqrt{d}}{\varphi^{k_0}}\right) &\mbox{ if } k=2k_0+1,\ k_0\in\mathbb{N}
			    \end{cases}\\
	\end{align*}
and
	\[
	R_k(u)=\pPell(u)+\begin{cases}
					A_k^{(1)} &\mbox{ if } 0\leq u<q_{k-1}\\
					A_k^{(2)} &\mbox{ if } q_{k-1}\leq u<q_{k}.
				 \end{cases}
	\]
Then
	\[
	\pPell(n)\in R_k(t_{\alpha}(n;k))+\mathbb{Z}.
	\]
\end{lemma}
\begin{proof}
By \eqref{qiLehmer} and the definition of $t_{\alpha}(n;k)$,
	\begin{align*}
	&n\varphi-t_{\alpha}(n;k)\varphi\\
&=\sum\limits_{\ell\geq \ceil{k/2}} b_{2\ell}(n)\varphi\left(\frac{m+\sqrt{d}}{2\sqrt{d}}\varphi^{\ell}-\frac{m-\sqrt{d}}{2\sqrt{d}}\varphi^{-\ell}\right)\\
		& +\sum\limits_{\ell\geq \ceil{(k-1)/2}} b_{2\ell+1}(n)\varphi\left(\frac{\varphi^{\ell+1}}{\sqrt{d}}-\frac{\varphi^{-\ell-1}}{\sqrt{d}}\right)\\
		&=\sum\limits_{\ell\geq \ceil{k/2}} b_{2\ell}(n)q_{2\ell+2}+\frac{m-\sqrt{d}}{2\sqrt{d}}\sum\limits_{\ell\geq \ceil{k/2}} b_{2\ell}(n)(\varphi^{-\ell-1}-\varphi^{-\ell+1})\\
		& +\sum\limits_{\ell\geq \ceil{(k-1)/2}} b_{2\ell+1}(n)q_{2\ell+3}+\frac{1}{\sqrt{d}}\sum\limits_{\ell\geq \ceil{(k-1)/2}} b_{2\ell+1}(n)(\varphi^{-\ell-2}-\varphi^{-\ell})\\
		&=\sum\limits_{\ell\geq \ceil{k/2}} b_{2\ell}(n)q_{2\ell+2}+\sum\limits_{\ell\geq \ceil{(k-1)/2}} b_{2\ell+1}(n)q_{2\ell+3}\\
		& +\frac{(-m+\sqrt{d})}{2}\sum\limits_{\ell\geq \ceil{k/2}} \frac{b_{2\ell}(n)}{\varphi^{\ell}}-\sum\limits_{\ell\geq \ceil{(k-1)/2}} \frac{b_{2\ell+1}(n)}{\varphi^{\ell+1}}.
	\end{align*}
Note that the first two terms in the above expression are integers. We first consider the case when $k$ is even. Write $k=2k_0$, $k_0\in\mathbb{N}$. Now
	\begin{align*}
	\sum\limits_{\ell\geq \ceil{(k-1)/2}} \frac{b_{2\ell+1}(n)}{\varphi^{\ell+1}}&\leq m\sum\limits_{\ell\geq k_0}\frac{1}{\varphi^{\ell+1}}=\frac{-m+\sqrt{d}}{2\varphi^{k_0}}
	\end{align*}	
since 
	\[
	\frac{1}{1-\varphi^{-1}}=\frac{2}{-m+\sqrt{d}}.
	\]
Suppose that $b_{k-1}(n)=0$. Then
	\[
	\frac{(-m+\sqrt{d})}{2}\sum\limits_{\ell\geq \ceil{k/2}} \frac{b_{2\ell}(n)}{\varphi^{\ell}}\leq \frac{(-m+\sqrt{d})}{2}\sum\limits_{\ell\geq k_0} \frac{1}{\varphi^{\ell}}=\frac{1}{\varphi^{k_0}}.
	\]
If $b_{k-1}(n)\neq 0$, then by condition $(iii)$ of Theorem \ref{Ostro}, $b_k(n)\neq 1$. Hence $b_k(n)=0$ and
	\begin{align*}
	\frac{(-m+\sqrt{d})}{2}\sum\limits_{\ell\geq \ceil{k/2}} \frac{b_{2\ell}(n)}{\varphi^{\ell}}&\leq\frac{(-m+\sqrt{d})}{2}\sum\limits_{\ell\geq k_0+1} \frac{1}{\varphi^{\ell}}=\frac{1}{\varphi^{k_0+1}}.
	\end{align*}
This proves the lemma when $k$ is even.
Next, we consider the case when $k$ is odd. Write $k=2k_0+1$, $k_0\in\mathbb{N}$. Then
	\begin{align*}
	&n(-\varphi)-t_{\alpha}(n;k)(-\varphi)\\
	&\equiv \frac{(m-\sqrt{d})}{2}\sum\limits_{\ell\geq k_0+1} \frac{b_{2\ell}(n)}{\varphi^{\ell}}+\sum\limits_{\ell\geq k_0} \frac{b_{2\ell+1}(n)}{\varphi^{\ell+1}}\Mod{1}.
	\end{align*}
Since
	\[
	\sum\limits_{\ell\geq k_0+1} \frac{b_{2\ell}(n)}{\varphi^{\ell}}\leq \frac{1}{\varphi^{k_0+1}}\frac{2}{-m+\sqrt{d}},
	\]
we get
	\[
	 \frac{(m-\sqrt{d})}{2}\sum\limits_{\ell\geq k_0+1} \frac{b_{2\ell}(n)}{\varphi^{\ell}}\geq -\frac{1}{\varphi^{k_0+1}}.
	\]
Suppose that $b_{k-1}(n)=0$. Then
	\[
	\sum\limits_{\ell\geq k_0} \frac{b_{2\ell+1}(n)}{\varphi^{\ell+1}}\leq \frac{m}{\varphi^{k_0+1}}\frac{2}{-m+\sqrt{d}}=\frac{-m+\sqrt{d}}{2\varphi^{k_0}}.
	\]
If $b_{k-1}(n)\neq 0$, then by condition $(iii)$ of Theorem \ref{Ostro}, $b_k(n)\neq m$. Hence 
	\begin{align*}
	\sum\limits_{\ell\geq k_0} \frac{b_{2\ell+1}(n)}{\varphi^{\ell+1}}&\leq \frac{m-1}{\varphi^{k_0+1}}+m\sum\limits_{\ell\geq k_0+1} \frac{1}{\varphi^{\ell+1}}\\
	&=\frac{m-1}{\varphi^{k_0+1}}+\frac{-m+\sqrt{d}}{2\varphi^{k_0+1}}=\frac{-m-1+\sqrt{d}}{\varphi^{k_0}}.
	\end{align*}
This proves the lemma when $k$ is odd.
\end{proof}
\begin{lemma}\label{part}
Fix an integer $k\geq 2$. The sets 
	\[
	R_k(u)+\mathbb{Z},\ 0\leq u<q_k,
	\]
form a partition of $\mathbb{R}$.
\end{lemma}
\begin{proof}
For each integer $u$ with $0\leq u<q_k$, let
	\[
	\tilde{R}_k(u)=R_k(u)\textrm{ mod } 1.
	\]
This set is the union of at most two intervals. Further, the
sum of the measures of the sets $\tilde{R}_k(u)$ is $1$. To prove this, we first consider the case when
$k$ is even. Write $k=2k_0$. Then, the sum of measures is
	\begin{align*}
	&\frac{2-m+\sqrt{d}}{2\varphi^{k_0}}q_{k-1}+\frac{1}{\varphi^{k_0}}(q_{k}-q_{k-1})\\
	&=\frac{2-m+\sqrt{d}}{2\varphi^{k_0}}q_{k-1}+\frac{1}{\varphi^{k_0}}((m-1)q_{k-1}+q_{k-2})\\
	&=\frac{m+\sqrt{d}}{2\varphi^{k_0}}q_{k-1}+\frac{1}{\varphi^{k_0}}q_{k-2}\\
	&=\frac{m+\sqrt{d}}{2\varphi^{k_0}}\left(\frac{1}{\sqrt{d}}\varphi^{k_0}-\frac{1}{\sqrt{d}}\varphi^{-k_0}\right)\\
	&+\frac{1}{\varphi^{k_0}}\left(\frac{m+\sqrt{d}}{2\sqrt{d}}\varphi^{k_0-1}-\frac{m-\sqrt{d}}{2\sqrt{d}}\varphi^{-(k_0-1)}\right)\\
	&=\frac{m+\sqrt{d}}{2\sqrt{d}}(1+\varphi^{-1})-\frac{1}{\varphi^{2k_0}}\left(\frac{m+\sqrt{d}}{2\sqrt{d}}+\frac{m-\sqrt{d}}{2\sqrt{d}}\varphi\right)\\
	&=1-0=1.
	\end{align*}
A similar calculation shows that the sum of measures also equals one when $k$ is odd.
Now we show that the sets $\tilde{R}_k(u)$ cover the interval $[0,1)$. If not, pick 
	\[
	x\in[0,1)\setminus \bigcup\limits_{0\leq u<q_k}\tilde{R}_k(u).
	\]
Then, there exists $\epsilon>0$ such that the sets $[x,x+\epsilon]$ and $\bigcup\tilde{R}_k(u)$ are disjoint. Since the sequence $(\{p_k(n)\})$ is dense in $[0,1)$, there is an integer $n_0$ such that $\{p_k(n_0)\}\in [x,x+\epsilon]$. Therefore 
	\[
	\{p_k(n_0)\}\notin \bigcup\limits_{0\leq u<q_k}\tilde{R}_k(u).
	\]
This contradicts Lemma \ref{crit1}. Thus the interval $[0,1)$ is the union of the sets $\tilde{R}_k(u)$. \\
Finally, we show that these sets are disjoint. If not, there exist $x,v,w$ with $v\neq w$, such that
	\[
	x\in \tilde{R}_k(v)\cap\tilde{R}_k(w).
	\]
Then there is an $\epsilon>0$ such that
	\[
	\lambda(\tilde{R}_k(v)\cap\tilde{R}_k(w))\geq \epsilon.
	\]
Thus
	\begin{align*}
	1&=\lambda\left(\bigcup\limits_{0\leq u<q_k}\tilde{R}_k(u)\right)\\
	&=\lambda\left((\tilde{R}_k(v)\setminus(\tilde{R}_k(v)\cap\tilde{R}_k(w)))\cup \bigcup\limits_{u\neq v}\tilde{R}_k(u) \right)\\
	&\leq \sum\limits_{0\leq u<q_k}\lambda(\tilde{R}_k(u))-\epsilon=1-\epsilon,
	\end{align*}
which is a contradiction. Therefore  the sets $\tilde{R}_k(u)$ must be disjoint.
\end{proof}
As an immediate consequence of Lemmas \ref{crit1} and \ref{part}, we get
\begin{corollary}\label{criteria}
Let $n\geq 0$, $k\geq 2$ and $0\leq u<q_k$. Then
	\[
	t_{\alpha}(n;k)=u
	\]
if and only if
	\[
	(-1)^kn\varphi\in R_k(u)+\mathbb{Z}.
	\]
\end{corollary}
We now present an inversion formula as in \eqref{dft1} for the function $e(\gamma S_{\alpha,k}(n))$.
\begin{lemma}\label{prop1}
Let $\gamma\in\mathbb{R}$ and $h,n\in\mathbb{Z}$ with $n\geq 0$.
Let $H,k$ be positive integers with $k\geq 2$. Define
	\begin{align*}
	M_k^{(1)}(h,\gamma)&=\sum\limits_{u<q_{k-1}}e(\gamma S_{\alpha}(u)-hp_k(u)),\\
	M_k^{(2)}(h,\gamma)&=\sum\limits_{q_{k-1}\leq u<q_k}e(\gamma S_{\alpha}(u)-hp_k(u)),
	\end{align*}
where $p_k(u)$ is as defined in Lemma \ref{crit1}. For $|h|\leq H$, there exist complex numbers $b_H^{(1)}(h), b_H^{(2)}(h), c_H^{(1)}(h)$ and $c_H^{(2)}(h)$ with
	\begin{align*}
	b_H^{(1)}(0)&=\begin{cases}
				\frac{2-m+\sqrt{d}}{2\varphi^{k_0}}&\mbox{if } k=2k_0,\ k_0\in\mathbb{Z}\\
				\frac{1}{\varphi^{k_0}}&\mbox{if } k=2k_0+1,\ k_0\in\mathbb{Z},
				\end{cases}\\
	b_H^{(2)}(0)&=\begin{cases}
					\frac{1}{\varphi^{k_0}}	&\mbox{if } k=2k_0,\ k_0\in\mathbb{Z}\\
					\frac{-m+\sqrt{d}}{2\varphi^{k_0}}	&\mbox{if } k=2k_0+1,\ k_0\in\mathbb{Z}
				  \end{cases}
	\end{align*}
and for $i=1,2$,	
	\begin{align*}			  
	|b_H^{(i)}(h)|&\leq\min\left(b_H^{(i)}(0),\frac{1}{|h|}\right)\textrm{ if } h\neq 0,\\
	|c_H^{(i)}(h)|&\leq2
	\end{align*}
such that
	\begin{align*}
	e(\gamma S_{\alpha,k}(n))&=\sum\limits_{i=1}^2\Big(\sum\limits_{|h|\leq H}b_H^{(i)}(h)e(hp_k(n))M_k^{(i)}(h,\gamma)\Big)\\
	&+O\Big( \frac{1}{H}\sum\limits_{|h|\leq H}c_H^{(1)}(h)e(hp_k(n)) \sum\limits_{u<q_{k-1}}e(-hp_k(u))\Big)\\
	&+O\Big( \frac{1}{H}\sum\limits_{|h|\leq H}c_H^{(2)}(h)e(hp_k(n)) \sum\limits_{q_{k-1}\leq u<q_k}e(-hp_k(u))\Big),
	\end{align*}
where the expression in the parantheses is a non-negative real number.
\end{lemma}
\begin{proof}
Let $u\in[0,q_k)$ and let
	\[
	\chi_u=\chi_{R_k(u)+\mathbb{Z}},
	\]
denote the indicator function of $R_k(u)+\mathbb{Z}$. Using Corollary \ref{criteria}, we get
	\begin{align*}
	e(\gamma S_{\alpha,k}(n))&=e(\gamma S_{\alpha}(t(n;k)))=\sum\limits_{u<q_k}e(\gamma S_{\alpha}(u))\chi_u(p_k(n))\\
				&=\sum\limits_{u<q_{k-1}}e(\gamma S_{\alpha}(u))\chi_u(p_k(n))+\sum\limits_{q_{k-1}\leq u<q_k}e(\gamma S_{\alpha}(u))\chi_u(p_k(n)).
	\end{align*}
Using Vaaler's \cite{Vaa} trigonometric polynomial approximation to the function $\{x\}-1/2$, one obtains (see \cite[Eqn. (5.7)]{Spt}),
	\begin{align*}
	\chi_{[a,b)+\mathbb{Z}}(x)&=\sum\limits_{|h|\leq H}a_H'(h)e(h(x-b))+O(\kappa_H(x-b)+\kappa_H(x-a)),
	\end{align*}
where
	\begin{align}\label{a_H}
	a_H'(0)=b-a,\ |a_H'(h)|\leq \min\left(b-a,\frac{1}{|h|}\right) \textrm{ if } h\neq 0
	\end{align}
and 
	\[
	\kappa_H(t)=\frac{1}{2(H+1)}\sum\limits_{|h|\leq H}\left(1-\frac{|h|}{H+1}\right)e(ht).
	\]
Further, $\kappa_H(t)$ is a non-negative real number for all real numbers $t$.
Write $R_k(u)=[p_k(u)+c,p_k(u)+d)$, where the values of $c$ and $d$ can be seen from Lemma \ref{crit1}. Then,
	\begin{align*}
	&\sum\limits_{u<q_{k-1}}e(\gamma S_{\alpha}(u))\chi_u(p_k(n))\\
	&=\sum\limits_{|h|\leq H}b_H^{(1)}(h)e(hp_k(n))\sum\limits_{u<q_{k-1}}e(\gamma S_{\alpha}(u)-hp_k(u))\\
	&+O\left(\sum\limits_{\beta\in\{c,d\}}\sum\limits_{u<q_{k-1}}\kappa_H(p_k(n)-p_k(u)-\beta)\right),
	\end{align*}
where
	\[
	b_H^{(1)}(h)=a_H'(h)e(-hd).
	\]
Using \eqref{a_H}, we get
	\[
	b_H^{(1)}(0)= d-c
	\]
and
	\[
	|b_H^{(1)}(h)|\leq \min\left(b_H^{(1)}(0),\frac{1}{|h|}\right)\textrm{ if } h\neq 0.
	\]
Let $\beta\in\{c,d\}$. Since $\kappa_H(t)$ is a non-negative real number for all real numbers $t$, we have
	\begin{align*}
	&|\sum\limits_{u<q_{k-1}}\kappa_H(p_k(n)-p_k(u)-\beta)|\\
	&\leq \frac{1}{H}\sum\limits_{|h|\leq H}\left(1-\frac{|h|}{H+1}\right)\sum\limits_{u<q_{k-1}}e(h(p_k(n)-p_k(u)-\beta)).
	\end{align*}
Thus, we obtain the term with $i=1$ claimed in the lemma with
	\[
	c_H^{(1)}(h)=\left(1-\frac{|h|}{H+1}\right)(e(-hc)+e(-hd)),
	\]
whose absolute value is at most $2$. The proof for the term with $i=2$ is similar.
\end{proof}
\begin{lemma}\label{matrix}
Let $\gamma\in\mathbb{R}$ with $\norm{m\gamma}\neq 0$. 
Then there exist $C,\eta>0$ such that for all $\beta\in\mathbb{R}$ and $k\geq 2$, we have
	\[
	\Big|\frac{1}{q_k}\sum\limits_{0\leq u<q_{k}}e(\gamma S_{\alpha}(u)+\beta u)\Big|\leq Ce^{-k\eta}.
	\]
\end{lemma}
\begin{proof}
Set 
	\[
	\mu_k=\mu_k(\gamma,\beta)=\frac{1}{q_k}\sum\limits_{0\leq u<q_{k}}e(\gamma S_{\alpha}(u)+\beta u),
	\]
	\[
	\tilde{M}_k=\max(|\mu_k|,|\mu_{k-1}|) \textrm{ and } c_k=\gamma+\beta q_k.
	\]
By \cite[Lemma 4 \& p. 333]{Fou}, $(\tilde{M}_k)$ is a decreasing sequence and for $k\geq 1$,
	\begin{align}
\nonumber	\tilde{M}_{k+2}&\leq \tilde{M}_{k-1}\left(1-\min\left(\frac{1}{12},\frac{3}{20}\left(1+\frac{1}{a_k}-\frac{1}{a_{k+1}}\right)a_{k+1}^2\norm{c_k}^2\right)\right)\\
\label{dif}	&\leq\tilde{M}_{k-1}\left(1-\min\left(\frac{1}{12},\frac{3}{20m}\norm{c_k}^2\right)\right)=:\tilde{M}_{k-1}\psi_k,
	\end{align}
as $a_i\in\{1,m\}$ for $i\geq 1$. Observe that
	\begin{align*}
	\norm{a_{k+3}\gamma}&= \norm{a_{k+3}\gamma+\beta(a_{k+3}q_{k+2}+q_{k+1}-q_{k+3})} \\
	&\leq \norm{c_{k+1}}+a_{k+3}\norm{c_{k+2}}+\norm{c_{k+3}}.
	\end{align*}
Thus
	\[
	\min(\norm{\gamma},\norm{m\gamma})\leq m\sum\limits_{1\leq i\leq 3}\norm{c_{k+i}},
	\]
implying that for every $k$, there exists $i\in\{1,2,3\}$ such that
	\[
	\norm{c_{k+i}}\geq\frac{\min(\norm{\gamma},\norm{m\gamma})}{3m}
	\]
and hence
	\begin{align*}
	\psi_{k+i}&\leq 1-\min\left(\frac{1}{12},\frac{\min(\norm{\gamma},\norm{m\gamma})^2}{60m^3}\right)\\
	&=1-\frac{\min(\norm{\gamma},\norm{m\gamma})^2}{60m^3}=:\psi.
	\end{align*}
Since $\norm{m\gamma}\neq 0$, $\psi<1$.
Fix $k$. Then there exists $k_i\in\{k-2,k-3,k-4\}$ such that $\psi_{k_i}\leq \psi$. Using \eqref{dif} and the fact that $(\tilde{M}_j)$ is a decreasing sequence, we get
	\[
	\tilde{M}_k\leq\tilde{M}_{k_i+2}\leq \tilde{M}_{k_i-1}\psi_{k_i}\leq \tilde{M}_{k-5}\psi.
	\]
Applying this repeatedly, we get
	\[
	\tilde{M}_k\leq \tilde{M}_5\psi^{k/5},
	\]
completing the proof of the lemma.
\end{proof}
\section{Proof of Theorem \ref{thm1}}\label{last}
With all the ingredients in place, the theorem follows as in \cite{Spt}. We include the details below.
Set
	\begin{align*}
	R&=\floor{N^{a-\eta a/(2\ln\varphi)}},\ H=\floor{N^{4a+\eta a/(2\ln\varphi)}},\\
	t&=\floor{a\ln N/\ln q},\ k=\floor{2a\ln N/\ln\varphi},
	\end{align*}
where $\eta$ is as in Lemma \ref{matrix}, $\varphi$ is as in \eqref{phin} and $a$ is sufficiently small.
Let
	\[
	g_{q}(n)=e(\theta S_{q}(n)),\ g_{q,t}(n)=e(\theta S_{q,t}(n)),
	\]
	\[
	g_{\alpha}(n)=e(\gamma S_{\alpha}(n)),\ g_{\alpha,k}(n)=e(\gamma S_{\alpha,k}(n))
	\]
and
	\[
	g(n)=g_{q}(n)g_{\alpha}(n).
	\]
We denote by $N'$ the largest multiple of $q^t$ not exceeding $N$.
By Lemmas \ref{corput} and \ref{trunc}, we get
	\begin{align}
\nonumber	&|\sum\limits_{n<N}g(n)|^2\\
\nonumber	&\ll\frac{N}{R}\sum\limits_{|r|<R}\Big(1-\frac{|r|}{R}\Big)\Big|\sum\limits_{0\leq n,n+r< N}g(n+r)\overline{g(n)}\Big|\\
\nonumber	&=\frac{N}{R}\sum\limits_{|r|<R}\Big(1-\frac{|r|}{R}\Big)\Big|\sum\limits_{0\leq n< N'}g_{q,t}(n+r)\overline{g_{q,t}(n)}g_{\alpha,k}(n+r)\overline{g_{\alpha,k}(n)}\Big|\\
\label{main_in}	&+O\left(NR+Nq^t+\frac{N^2R}{q^{t}}+\frac{N^2R}{q_{k-1}}\right).
	\end{align}
We use the expressions for $g_{q,t}(n+r), g_{q,t}(n), g_{\alpha,k}(n+r), g_{\alpha,k}(n)$ from \eqref{dft1}, \eqref{dft2} and Lemma \ref{prop1}. In the product $g_{\alpha,k}(n+r)g_{\alpha,k}(n)$, there are sixteen summands of the following three kinds: four products of the main terms in the expressions for $g_{\alpha,k}(n+r)$ and $g_{\alpha,k}(n)$, eight products of the main terms and error terms, four products of the error terms. We now consider these three cases separately.\\\\
\textbf{Case I.} \textit{(Summands with both factors as main terms)}\\
Let $h=h_1+h_2$ and $\ell=\ell_1+\ell_2$. We need to estimate
	\begin{align*}
	&\frac{N}{R^2}\sum\limits_{\substack{\ell_1,\ell_2<q^t\\ |h_1|,|h_2|\leq H}}G_t(\ell_1)\overline{G_t(-\ell_2)}b_H^{(i)}(h_1)\overline{b_H^{(j)}(-h_2)}M_k^{(i)}(h_1,\gamma)\overline{M_k^{(j)}(-h_2,\gamma)}\\
	&\times \sum\limits_{|r|<R}(R-|r|)e\Big(r\Big(\frac{\ell_1}{q^t}+h_1(-1)^k\varphi\Big)\Big)\sum\limits_{n<N'}e\Big(n\Big(\frac{\ell}{q^t}+(-1)^kh\varphi\Big)\Big)
	\end{align*}
We first consider the subcase when $h=0$. If $\ell\not\equiv 0\pmod{q^t}$, then $\sum\limits_{n<N'}e\Big(n\ell q^{-t}\Big)\Big)=0$ as $q^t|N'$. Hence we assume that $\ell\equiv 0\pmod{q^t}$. We estimate the sum over $n$ trivially. By \eqref{Parseval} and Lemmas \ref{elem}, \ref{prop1}, \ref{matrix} and \ref{discrep}, we find that the above expression is
	\begin{align*}
	&\ll\frac{N^2q_{k}}{e^{\eta k}R^2}\sum\limits_{\ell_1<q^t}|G_t(\ell_1)|^2\sum\limits_{|h_1|\leq H}\min\Big(\frac{1}{|h_1|},\frac{1}{\varphi^{k_0}}\Big)\min\Big(R^2, \left\Vert\frac{\ell_1}{q^{t}}+h_1(-1)^k\varphi \right\Vert^{-2}\Big)\\
	&\ll\frac{N^2q_{k}}{e^{\eta k}R^2}\sup\limits_{\ell_1\in\mathbb{Z}} \sum\limits_{|h|\leq H} \min\Big(\frac{1}{|h|},\frac{1}{q_k}\Big)\min\Big(R^2,\norm{\ell_1q^{-t}+|h|\varphi}^{-2}\Big)\\
	&\ll\frac{N^2q_{k}}{e^{\eta k}R^2}\sum\limits_{s< H/q_k} \min\Big(\frac{1}{sq_k},\frac{1}{q_k}\Big)\sup\limits_{u\in\mathbb{R}}\sum\limits_{h< q_k}\min\Big(R^2,\norm{u+(sq_k+h)\varphi}^{-2}\Big)\\
	&\ll\frac{N^2}{e^{\eta k}R^2}\ln H(Rq_k+R^2\ln q_k)=N^2e^{-\eta k}\ln H(q_kR^{-1}+\ln q_k).
	\end{align*}
	Next, we consider the case when $h\neq 0$. Since $\varphi$ is badly approximable, there is a constant $c_1$ such that for all integers $\ell',h'$ with $h'\neq 0$
		\[
		\big|\varphi+\frac{\ell'}{h'q^t}\big|>\frac{c_1}{(h'q^{t})^2}
		\]
		implying that
		\be\label{scr}
		\norm{h\varphi+\ell q^{-t}}>c_1(hq^{2t})^{-1}.
		\ee
We estimate $G_t$ and the sum over $r$ trivially, and use Lemma \ref{elem} together with \eqref{scr} to obtain the following upper bound.
		\begin{align*}
		&Nq^{2t}\sum\limits_{\substack{|h_1|,|h_2|\leq H\\ h_1+h_2\neq 0}}\sup\limits_{\ell\in\mathbb{Z}}\Big|\sum\limits_{n<N'}e(n(\ell q^{-t}+(-1)^kh\varphi))\Big|\\
		&=Nq^{2t}\sum\limits_{1\leq |h|\leq 2H}(2H+1-|h|)\sup\limits_{\ell\in\mathbb{Z}}\Big|\sum\limits_{n<N'}e(n(\ell q^{-t}+h\varphi))\Big|\\
		&\ll NHq^{2t}\sum\limits_{1\leq |h|\leq 2H}hq^{2t}\ll NH^3q^{4t}.
		\end{align*}
\\
\textbf{Case II.} \textit{(Summands with exactly one factor as main term)}\\
Suppose that the main term comes from the expression for $g_{\alpha,t}(n+r)$. (The other case is similar.) Then, we need to estimate
	\begin{align*}
	&\frac{N}{R}\sum\limits_{|r|<R}\Big(1-\frac{|r|}{R}\Big)\Big|\sum\limits_{n<N'}\sum\limits_{\substack{\ell_1,\ell_2<q^t\\|h_1|\leq H}}e(\ell_1(n+r)q^{-t}+\ell_2nq^{-t}+h_1p_k(n+r))\\
	&G_t(\ell_1,\theta)\overline{G_t(-\ell_2,\theta)}
	b_H^{(i)}(h_1)M_k^{(i)}(h_1,\gamma)\\
	&\times O\Big( \frac{1}{H}\sum\limits_{|h_2|\leq H}c_H^{(j)}(h_2)e(h_2p_k(n)) \sum\limits_{u}e(-h_2p_k(u))\Big)\Big|.
	\end{align*}
Recall that the expression in the error term is a non-negative real number. We estimate $G_t$ and the sum over $r$ trivially. Further, using Lemmas \ref{prop1}, \ref{elem} and \ref{discrep}, we obtain the following upper bound.
	\begin{align*}
	&\frac{Nq^{2t}q_k}{H}\sum\limits_{|h_1|\leq H}\min\Big(b_H^{(i)}(0),\frac{1}{|h_1|}\Big)\\
	&\times \sum\limits_{|h_2|\leq H}\big|\sum\limits_{u}e(-h_2p_k(u))\big| \big|\sum\limits_{n<N'}e(h_2p_k(n))\big| \\
	&\ll\frac{Nq^{2t}q_k\ln H}{H}\sum\limits_{|h_2|\leq H}  \min\left(q_k,\frac{1}{||h_2\varphi||}\right)\min\left(N,\frac{1}{||h_2\varphi||}\right)\\
	&\ll\frac{Nq^{2t}q_k\ln H}{H}\sum\limits_{|h_2|\leq H}  \min\left(q_kN,\frac{1}{||h_2\varphi||^2}\right)\\
	&\ll N^{3/2}q^{2t}q_k^{3/2}\ln H+N^2q^{2t}q_k^2\frac{(\ln H)^2}{H}.
	\end{align*}
\textbf{Case III.} \textit{(Summands with both factors as error terms)}\\
We estimate one of the error terms trivially by $q_k$. Proceeding as in Case II, we get
	\begin{align*}
	&\frac{N}{R}\sum\limits_{|r|<R}\Big(1-\frac{|r|}{R}\Big)\Big|\sum\limits_{n<N'}\sum\limits_{\substack{\ell_1,\ell_2<q^t}}e(\ell_1(n+r)q^{-t}+\ell_2nq^{-t})\\
	&G_t(\ell_1,\theta)\overline{G_t(-\ell_2,\theta)}\\
	&\times O\Big( \frac{1}{H}\sum\limits_{|h_1|\leq H}c_H^{(i)}(h_1)e(h_1p_k(n+r)) \sum\limits_{u}e(-h_1p_k(u))\Big)\Big|\\
	&\times O\Big( \frac{1}{H}\sum\limits_{|h_2|\leq H}c_H^{(j)}(h_2)e(h_2p_k(n)) \sum\limits_{u}e(-h_2p_k(u))\Big)\Big|\\
	&\ll\frac{Nq^{2t}q_k}{H}\sum\limits_{|h_2|\leq H}\big|\sum\limits_{u}e(-h_2p_k(u))\big| \big|\sum\limits_{n<N'}e(h_2p_k(n))\big| \\
	&\ll N^{3/2}q^{2t}q_k^{3/2}+N^2q^{2t}q_k^2\frac{\ln H}{H}.
	\end{align*}
	Combining the three cases and \eqref{main_in}, we get
	\[
	|\sum\limits_{n<N}g(n)|^2\ll N^{2-2\delta}
	\]
	for some $\delta>0$.\qed\\\\
\noindent\textbf{Proof of Corollary \ref{cor1}}\\
We first recall an estimate of Gel$\cprime$fond \cite{Ge}: If $k_1,m_1,q$ are positive integers with $m_1\geq 2$, $ k_1<m_1$ and $\gcd(m_1,q-1)=1$, then there exists $\delta_1>0$ such that for every $a\in\mathbb{R}$, we have
	\be\label{Gel}
	\sum\limits_{n<N}e\left(an+\frac{k_1}{m_1}S_q(n)\right)= O(N^{1-\delta_1}).
	\ee
Now, since
	\[
	\frac{1}{b}\sum\limits_{0\leq \ell<b}e\left(\frac{a}{b}\ell\right)=1\textrm{ or }0
	\]
according to whether $b$ divides $a$ or not, we have
\begin{align*}
&|\{0\leq n<N: S_{q}(n)\equiv a_1\Mod{m_1},\ S_{\alpha}(n)\equiv a_2\Mod{m_2}\}|\\
&=\sum\limits_{n<N}\frac{1}{m_1}\sum\limits_{0\leq k_1<m_1}e\left(k_1\frac{S_q(n)-a_1}{m_1}\right)\frac{1}{m_2}\sum\limits_{0\leq k_2<m_2}e\left(k_2\frac{S_{\alpha}(n)-a_2}{m_2}\right)\\
&=\frac{1}{m_1m_2}\sum\limits_{\substack{0\leq k_1<m_1\\0\leq  k_2<m_2}}e\left(-\frac{k_1a_1}{m_1}-\frac{k_2a_2}{m_2}\right)\sum\limits_{n<N}e\left(\frac{k_1}{m_1} S_{q}(n)+\frac{k_2}{m_2} S_{\alpha}(n)\right)\\
&=\frac{N}{m_1m_2}+O\Big(\frac{1}{m_1m_2}\sum\limits_{1\leq k_1<m_1}\big|\sum\limits_{n<N}e\left(\frac{k_1}{m_1}S_q(n)\right)\big|\\
&+\frac{1}{m_1m_2}\sum\limits_{\substack{0\leq k_1<m_1\\ 1\leq k_2<m_2}} \Big|\sum\limits_{n<N} e\left(\frac{k_1}{m_1} S_{q}(n)+\frac{k_2}{m_2} S_{\alpha}(n)\right)\Big| \Big)\\
&=\frac{N}{m_1m_2}+O(N^{1-\delta'}),
\end{align*}
where the last equality follows from \eqref{Gel} and Theorem \ref{thm1}.\qed\\
\section*{Acknowledgement}
The author thanks the University of Waterloo, Canada for its kind hospitality during the writing of this paper.


\end{document}